\documentclass[11pt]{article}

\usepackage{amsmath}
\usepackage{amssymb}
\usepackage{amscd}
\usepackage{amsthm}
\usepackage{indentfirst}

\newcommand{\Z}{\ensuremath{\mathbb{Z}}}
\newcommand{\Q}{\ensuremath{\mathbb{Q}}}

\input cyracc.def \font\tencyr=wncyr10 \def\russe{\tencyr\cyracc} 
\def\Sha{\text{\russe{Sh}}}

\usepackage[margin=2.5cm]{geometry}

\theoremstyle{plain}
\newtheorem{thm}{Theorem}[section]
\newtheorem{lem}[thm]{Lemma}
\newtheorem{cor}[thm]{Corollary}
\newtheorem{prop}[thm]{Proposition}

\theoremstyle{definition}
\newtheorem{dfn}[thm]{Definition}

\newtheorem{rmq}[thm]{Remark}

\DeclareMathOperator{\Hom}{Hom}
\DeclareMathOperator{\Spec}{Spec}
\DeclareMathOperator{\Pic}{Pic}
\DeclareMathOperator{\Gal}{Gal}

\DeclareMathOperator{\Cl}{Cl}

\DeclareMathOperator{\Sel}{Sel}
\DeclareMathOperator{\rk}{rk}

\newcommand{\A}{\mathcal{A}}
\newcommand{\Seq}{\mathcal{S}}

\newcommand{\Gm}{\mathbf{G}_{{\rm m}}}

\begin{document}

\title{On the splitting of the Kummer exact sequence}

\author{Jean Gillibert \and Pierre Gillibert}

\date{April 2018}

\maketitle

\begin{abstract}
We prove the splitting of the Kummer exact sequence and related exact sequences in arithmetic geometry.
\end{abstract}


\section{Introduction}


Exact sequences ``of Kummer type'' are ubiquitous in algebraic and arithmetic geometry. Let us start with the most famous one: if $X$ is a scheme, and if $m>0$ is an integer, the $m$-th power map on the multiplicative group $\Gm$ is surjective for the fppf topology on $X$, with kernel $\mu_m$, and induces the well-known Kummer exact sequence
\begin{equation}
\label{kummer0}
\begin{CD}
0 @>>> \Gamma(X,\mathcal{O}_X^{\times})/m @>>> H^1(X,\mu_m) @>>> \Pic(X)[m] @>>> 0. \\
\end{CD}
\end{equation}

In this paper, we prove that \eqref{kummer0} splits. The analogous statement for abelian varieties was proved in \cite[Prop.~6.7]{BKLP15}, as it was pointed to us by Bjorn Poonen. In fact, we prove a more general theorem containing both of these results as special cases. Our main result, Theorem~\ref{coh_splitting}, is purely cohomological: it applies to derived functors between abelian categories, and concerns Kummer exact sequences in all degrees.

We also discuss the counterpart of this generality. More precisely, we prove that there is no canonical splitting of \eqref{kummer0}, in the sense that one cannot choose splittings so as to commute with base change  (\S{}\ref{subsect_BC}), or Galois actions (\S{}\ref{subsect_Chris}).


\subsection{Splitting of cohomological exact sequences}

If $M$ is an abelian group (or, more generally, an object in some abelian category), and if $m>0$ is an integer, we denote by $[m]:M\to M$ the multiplication by $m$ map on $M$, by $M[m]$ its kernel, by $mM$ its image, and by $M/m$ its cokernel.

We say that $M$ is \emph{$m$-torsion} if $mM=0$. We say that $M$ is a torsion (abelian) group of \emph{bounded exponent} if $M$ is $m$-torsion for some integer $m>0$.

We now state our main result in the setting of derived functors between abelian categories.

\begin{thm}
\label{coh_splitting}
Let $\mathbf{C}$ be an abelian category with enough injectives, and let $F$ be a left exact functor from $\mathbf{C}$ to the category of abelian groups. For $i\geq 0$, we denote by $R^iF$ the $i$-th right derived functor of $F$.

Let $N$ be an object of $\mathbf{C}$, and let $m>0$ be an integer such that $[m]:N\to N$ is an epimorphism. Then for all $i>0$ the short ``Kummer'' exact sequence induced by $[m]$ on cohomology:
\begin{equation}
\label{cohomo_seq}
\begin{CD}
0 @>>> R^{i-1}F(N)/m @>>> R^iF(N[m]) @>>> R^iF(N)[m] @>>> 0 \\
\end{CD}
\end{equation}
splits.
\end{thm}

Our proof proceeds as follows: according to the Chinese remainder theorem, we may assume that $m=p^n$ for some prime $p$. By considering the natural diagram relating the Kummer sequence for $p^k$ with the one for $p^{k+1}$ (see the proof of Theorem~\ref{coh_splitting}), we prove (in our ``Main Lemma'', Lemma~\ref{MainLemma}) that the Kummer exact sequence is \emph{pure exact} (cf. Definition~\ref{pure_exact_dfn}). The splitting of the Kummer exact sequence ultimately relies on the first Pr{\"u}fer theorem, which states that any torsion abelian group of bounded exponent is a direct sum of cyclic groups.

We shall now describe the consequences and limitations of Theorem~\ref{coh_splitting}.


\subsection{Kummer theory on semi-abelian schemes}

In order to consider at the same time Kummer theory on tori and on abelian varieties, we state our results in the setting of semi-abelian schemes.


Taking $N$ in Theorem~\ref{coh_splitting} to be the fppf sheaf associated to a semi-abelian scheme (see \cite[Definition~2.3]{FC}) yelds the following.

\begin{cor}
\label{prop1}
Let $\A$ be a semi-abelian scheme over a scheme $X$. Then for all $m>0$ and for all $i>0$, the exact sequence
\begin{equation}
\label{kummer_i}
\begin{CD}
0 @>>> H^{i-1}(X,\A)/m @>>> H^i(X,\A[m]) @>>> H^i(X,\A)[m] @>>> 0 \\
\end{CD}
\end{equation}
splits.
\end{cor}

\begin{rmq}
\label{rmq1}
\begin{enumerate}
\item If $\A$ is the multiplicative group $\Gm$, then for $i=1$ we recover the exact sequence \eqref{kummer0} from the beginning. The result also applies to abelian schemes over $X$, and in particular to abelian varieties over any field.
\item Our original motivation for the splitting of the Kummer exact sequence arised in a joint paper by Aaron Levin and the first author \cite{GL12}. The setting was the following: $K$ is a number field, $\mathcal{O}_K$ is its ring of integers, and $m$ is fixed. The Kummer exact sequence for $\Gm$ over $\Spec(\mathcal{O}_K)$ reads
\begin{equation}
\label{OriginalKummer_seq}
\begin{CD}
0 @>>> \mathcal{O}_K^{\times}/m @>>> H^1(\Spec(\mathcal{O}_K),\mu_m) @>>> \Cl(K)[m] @>>> 0 \\
\end{CD}
\end{equation}
where $\Cl(K)$ denotes the ideal class group of $K$. Using elementary considerations on torsion abelian groups, it was proved in \cite{GL12} that
$$
\rk_m H^1(\Spec(\mathcal{O}_K),\mu_m)\geq \rk_m \Cl(K)[m]+\rk_m (\mathcal{O}_K^{\times}/m)
$$
where $\rk_m M$ denotes the largest integer $r$ such that $M$ has a subgroup isomorphic to $(\Z/m)^r$. Now that the splitting of the sequence \eqref{OriginalKummer_seq} is proved, the inequality above turns out to be an equality.
\item Let $X$ be a connected Dedekind scheme, and let $A$ be an abelian variety over its function field. If $A$ has semi-stable reduction over $X$, then the connected component of the N{\'e}ron model of $A$ is a semi-abelian scheme over $X$, so Corollary~\ref{prop1} applies to it.
\item Let $A$ be an abelian variety over a number field $K$. Fix a positive integer $m$. Define $\Sel^m(A/K)$ and $\Sha(A/K)$ as in \cite[X.4]{silverman}. Pulling back the split exact sequence \eqref{kummer_i} for $A/K$ and $i=1$ by the inclusion $\Sha(A/K)[m]\hookrightarrow H^1(K,A)[m]$ yelds a split exact sequence
\begin{equation}
\label{mdescent}
\begin{CD}
0 @>>> A(K)/m @>>> \Sel^m(A/K) @>>> \Sha(A/K)[m] @>>> 0. \\
\end{CD}
\end{equation}
which is none other than the ``$m$-descent exact sequence''.
\end{enumerate}
\end{rmq}


\subsection{A remark on base change}
\label{subsect_BC}

The splitting of the sequence \eqref{kummer0} is in general not compatible with base change, as shows the following example. Let $k$ be an algebraically closed field, and let $m>0$ coprime to the characteristic of $k$. Let $C$ be a smooth, connected, projective $k$-curve of genus $g\geq 1$. Then $\Pic(C)[m]\simeq (\Z/m)^{2g}$ and $\Gamma(C,\mathcal{O}_C^{\times})/m=k^{\times}/m=0$. When trying to compare the Kummer exact sequence \eqref{kummer0} over $C$ with the same sequence over $\Spec(k(C))$, one gets the following commutative diagram
$$
\begin{CD}
@. 0 @>>> H^1(C,\mu_m) @>\sim>> \Pic(C)[m] @>>> 0 \\
@. @VVV @VVV @VVV \\
0 @>>> k(C)^{\times}/m @>\sim>> H^1(\Spec(k(C)),\mu_m) @>>> 0 \\
\end{CD}
$$
in which the vertical maps are obtained by restriction to the generic point of $C$. The category of smooth, connected, projective $k$-curves being equivalent to that of function fields (of one variable) over $k$, the vertical map in the middle is injective. Therefore, the splitting of the two sequences does not commute with the vertical base change map.


\subsection{A remark on Galois actions}
\label{subsect_Chris}

Let us consider a finite Galois extension $K/k$ of fields of characteristic $0$, and let $E$ be an elliptic curve defined over $k$. Then the group $G:=\Gal(K/k)$ acts on $E(K)$ and $H^1(K,E)$. Consequently, if $p$ is a prime number, the Kummer exact sequence
\begin{equation}
\label{chris_seq}
\begin{CD}
0 @>>> E(K)/p @>>> H^1(K,E[p]) @>>> H^1(K,E)[p] @>>> 0 \\
\end{CD}
\end{equation}
is an exact sequence of $\mathbb{F}_p[G]$-modules. It is a natural question to ask if this exact sequence splits in the category of $\mathbb{F}_p[G]$-modules, i.e. if it admits a $G$-equivariant section.

The following was communicated to us by Christian Wuthrich.

\begin{prop}
Let $p\neq 2$ be a prime number, and let $K/\Q_p$ be the unique unramified Galois extension of $\Q_p$ whose Galois group $G$ is cyclic of order $p$. Then there exists an elliptic curve $E$ defined over $\Q_p$ such that \eqref{chris_seq} does \emph{not} split as a sequence of $\mathbb{F}_p[G]$-modules.
\end{prop}

\begin{proof}
Because  $p\neq 2$ and $K/\Q_p$ is unramified, we have that $\mu_p(K)=\mu_p(\Q_p)=\{1\}$. The multiplicative group $\Z_p^{\times}$ can be described as
$$
\Z_p^{\times}\simeq \mathbb{F}_p^{\times}\times (1+p\Z_p)\simeq \Z/(p-1)\times \Z_p.
$$
Let us choose some $u\in\Z_p^{\times}$ which is not a $p$-th power. We claim that $u$ is not a $p$-th power in $K^{\times}$ either. Indeed, if it were the case, then one could write $u=\alpha^p$ for some $\alpha\in K$, and by Galois theory all conjugates of $\alpha$ would belong to $K$, hence $p$-th roots of unity would belong to $K$, contradiction.

Let $q:=p^pu$, and let $E/\Q_p$ be the Tate elliptic curve with parameter $q$. Then $E(K)\simeq K^{\times}/q^{\Z}$, hence $E(K)[p]=0$ because $\mu_p(K)=\{1\}$ and $q$ is not a $p$-th power in $K^{\times}$.

By applying $G$-invariants to the sequence \eqref{chris_seq}, we obtain the bottom exact sequence in the commutative diagram below
\begin{equation}
\label{chris_diagram}
\begin{CD}
0 @>>> E(\Q_p)/p @>>> H^1(\Q_p,E[p]) @>>> H^1(\Q_p,E)[p] @>>> 0 \\
@. @VVV @V\beta VV @V\gamma VV \\
0 @>>> (E(K)/p)^G @>>> H^1(K,E[p])^G @>\rho>> H^1(K,E)[p]^G \\
\end{CD}
\end{equation}
where $\rho$ denotes the restriction of the natural map $H^1(K,E[p])\to H^1(K,E)[p]$ to $G$-invariants subgroups on both sides.
As noted above, $E(K)[p]=0$, hence $\beta$ is an isomorphism, according to the inflation-restriction exact sequence.
 
If \eqref{chris_seq} were split as a sequence of $\mathbb{F}_p[G]$-modules, then the map $\rho$ would be surjective. On the other hand, if $\rho$ were surjective then according to the diagram \eqref{chris_diagram} above, $\gamma$ would also be surjective. We shall now prove that this is not the case.

According to local Tate duality \cite{tate}, the Pontryagin dual of $\gamma$ can be identified with the ``norm'' map
$$
(E(K)/p)_G \longrightarrow E(\Q_p)/p
$$
where $(E(K)/p)_G$ denotes $G$-coinvariants. The kernel of this map is (by definition) Tate's cohomology group $\hat{H}^{-1}(G,E(K)/p)$. The group $G$ being cyclic, this group is isomorphic to $H^1(G,E(K)/p)$. So in order to prove that $\gamma$ is not surjective it suffices to prove that $H^1(G,E(K)/p)$ is not zero.

Let $\mathcal{O}_K$ be the ring of integers of $K$. Because $q$ has $p$-adic valuation $p$, we have an exact sequence
\begin{equation}
\label{infres}
\begin{CD}
0 @>>> \mathcal{O}_K^{\times} @>>> K^{\times}/q^{\Z} @>>> \Z/p @>>> 0 \\
\end{CD}
\end{equation}
where the map on the right is induced by the $p$-adic valuation. Note that $G$ acts trivially on the quotient group. According to \cite[Proposition~7.1.2]{neukirch}, the extension $K/\Q_p$ being unramified, $\mathcal{O}_K^{\times}$ is $G$-acyclic. Therefore, we have isomorphisms $H^i(G,E(K))\simeq H^i(G,K^{\times}/q^{\Z})\simeq H^i(G,\Z/p)$ for all $i>0$. In particular, $H^1(G,E(K))$ and $H^2(G,E(K))$ are both cyclic of order $p$.

Because $E(K)$ has trivial $p$-torsion, we have an exact sequence
$$
\begin{CD}
0 @>>> E(K) @>[p]>> E(K) @>>> E(K)/p @>>> 0. \\
\end{CD}
$$
The group $G$ being cyclic of order $p$, the groups $H^i(G,E(K))$ are $p$-torsion for $i>0$, hence the exact sequence above induces on cohomology a short exact sequence
$$
\begin{CD}
0 @>>> H^1(G,E(K)) @>>> H^1(G,E(K)/p) @>>> H^2(G,E(K)) @>>> 0. \\
\end{CD}
$$
This proves that $H^1(G,E(K)/p)$ is isomorphic to $(\Z/p)^2$, hence the result.
\end{proof}

%


\section{Proofs}


\subsection{The Main Lemma}

The aim of this section is to prove the following:

\begin{lem}[Main Lemma]
\label{MainLemma}
Let $p$ be a prime number, let $n>0$ be an integer, and let $C$ be an abelian group.

Assume that we have, for each $0<k\leq n$, an exact sequence $\Seq_k$ of $\Z/p^k$-modules
\begin{equation}
\label{Sk}
\tag{$\Seq_k$}
\begin{CD}
0 @>>> A_k @>>> B_k @>>> C[p^k] @>>> 0 \\
\end{CD}
\end{equation}
together with a sequence of maps between these sequences
\begin{equation}
\label{S2}
\begin{CD}
0 @>>> A_k @>>> B_k @>>> C[p^k] @>>> 0 \\
@. @VVV @VVV @VVV \\
0 @>>> A_{k+1} @>>> B_{k+1} @>>> C[p^{k+1}] @>>> 0 \\
\end{CD}
\end{equation}
where the vertical map on the right is the canonical inclusion $C[p^k]\hookrightarrow C[p^{k+1}]$.

Then the exact sequence $\Seq_n$ splits.
\end{lem}

Before we prove this result, we briefly recall the definition of pure exact sequences in the category of abelian groups. For a detailed exposition, we refer to \cite[Chap.~V]{fuchs70}.

\begin{dfn}
\label{pure_exact_dfn}
A short exact sequence of abelian groups
\begin{equation}
\label{pure_seq1}
\begin{CD}
0 @>>> A @>>> B @>>> C @>>> 0 \\
\end{CD}
\end{equation}
is a \emph{pure exact sequence} if, for all $c\in C$, there exists $b\in B$ such that $b\mapsto c$ and $b$ has the same order than $c$.
\end{dfn}

\begin{rmq}
\begin{enumerate}
\item Equivalently, the sequence \eqref{pure_seq1} is pure exact if and only if, for every integer $n\geq 0$, $nA=A\cap (nB)$;
\item Split exact sequences are pure, but the converse is false (see \S{}\ref{exemplePierre} below).
Nevertheless, any pure exact sequence is a direct limit of a system of split exact sequences, see \cite[Chap.~V,~\S{}29]{fuchs70}.
\end{enumerate}
\end{rmq}

\begin{lem}
\label{pureSplit}
Assume that $C$ is a direct sum of cyclic groups. Then any pure exact sequence of the form \eqref{pure_seq1} splits.
\end{lem}

\begin{proof}
This is proved in \cite[Chap.~V,~\S{}28]{fuchs70}, but for the reader's convenience we sketch a proof: assume $C$ is cyclic and the sequence \eqref{pure_seq1} is pure exact. Pick a generator $c$ of $C$. By purity, there exists $b\in B$ which maps to $c$ and has the same order than $C$. It follows that the map $C\to B;c^n\mapsto b^n$ is a group-theoretic section of the map $g:B\to C$. The general case follows: if $C=\oplus_{i\in I} C_i$ is direct sum of cyclic groups, then there exist partial sections $s_i:C_i\to g^{-1}(C_i)$ which give rise to a section $\oplus s_i:C\to B$, hence the result.
\end{proof}

\begin{proof}[Proof of Lemma~2.1]
For each $1\le k\le n$, let us denote by $g_k: B_k\to C[p^k]$ the surjective morphism in $\Seq_k$.

According to the first Pr{\"u}fer theorem (see \cite[Chap.~III,~\S{}17]{fuchs70}), any torsion abelian group of bounded exponent is a direct sum of cyclic groups. It follows that $C[p^n]$ is a direct sum of cyclic groups, hence, according to Lemma~\ref{pureSplit}, in order to prove that $\Seq_n$ splits it suffices to show that $\Seq_n$ is pure exact.

Let $t\in C[p^n]\setminus \{0\}$, then the order of $t$ is $p^k$ for some $k>0$, thus $t\in C[p^k]$. By surjectivity of $g_k$, there exists $x\in B_k$ such that $g_k(x)=t$, and the order of $x$ is $p^k$, because $B_k$ is a $p^k$-torsion group. Let $y$ be the image of $x$ in $B_n$. Then $y$ has order dividing $p^k$, and $g_n(y)=t$ has order $p^k$, hence it follows that $y$ has order $p^k$. This shows that $\Seq_n$ is pure exact.
\end{proof}


\subsection{Direct limits}
\label{exemplePierre}

One may ask if, in the situation of Lemma~\ref{MainLemma}, the exact sequence of direct limits
\begin{equation}
\label{DLseq}
\begin{CD}
0 @>>> \varinjlim A_k @>>> \varinjlim B_k @>>> C[p^{\infty}] @>>> 0 \\
\end{CD}
\end{equation}
does split. The answer is negative, as the following example shows.

Given an integer $n\ge 1$, set $B_n:=\prod_{k=1}^n\Z/p^k$, set $C_n:=\Z/p^n$, and consider $g_n: B_n\to C_n$ defined by $g_n((x_k+p^k\Z)_{1\le k\le n}) = \sum_{k=1}^n p^{n-k}x_k + p^n\Z$. Note that $g_n$ is surjective. Set $A_n=\ker g_n$, and denote $f_n: A_n\to B_n$ the inclusion. Thus $0\to A_n\to B_n\to C_n\to 0$ is an exact sequence.

Set $\psi_n: B_n\to B_{n+1}$, defined by $\psi_n(\vec x) = (\vec x,0)$, set $\eta_n: C_n\to C_{n+1}$, defined by $\eta_n(x+p^{n}\Z) = px+p^{n+1}\Z$. Note that $\eta_n\circ g_n =g_{n+1}\circ \psi_n$. Let $x\in A_n=\ker g_n$, it follows that $g_{n+1}(\psi_n(x)) = \eta_n(g_n(x))=\eta_n(0)=0$, therefore $\psi_n(x)\in A_{n+1}=\ker g_{n+1}$. Denote by $\varphi_n: A_n\to A_{n+1}$ the restriction of the map $\psi_n$. The following diagram is commutative

\[
\begin{CD}
0 @>>> A_n @>f_n>> B_n @>g_n>> C_n @>>> 0 \\
@. @V\varphi_nVV @V\psi_nVV @V\eta_nVV \\
0 @>>> A_{n+1} @>f_{n+1}>> B_{n+1} @>g_{n+1}>> C_{n+1} @>>> 0. \\
\end{CD}
\]

Note that for each $n$, the morphism $g_n$ has a section. However, if one restricts a section of $g_{n+1}$ to the subgroup $C_n$ it cannot factor through the map $\psi_n$, in other words there is no compatible family of sections.

Let us denote by
\begin{equation}
\label{lim_seq}
\begin{CD}
0 @>>> A_\infty @>>> B_\infty @>>> C_\infty @>>> 0 \\
\end{CD}
\end{equation}
the direct limit of the family of exact sequences above. For each $n\ge 1$, we have a natural identification of $C_\infty[p^n]$ with $C_n$. In $C_\infty$ there exist nonzero elements which are infinitely divisible by $p$ (indeed, all elements are divisible by $p$). However, there is no element in $B_\infty$ which is infinitely divisible by $p$. Therefore there is no section $s: C_\infty\to B_\infty$. We note that the limit sequence \eqref{lim_seq} is pure exact.

Nevertheless one may prove, under some additional hypothesis, that the exact sequence of direct limits splits.

\begin{prop}
\label{DirectLimit}
Let us consider an infinite family $(\Seq_k)_{k>0}$ of short exact sequences and maps between these sequences, as in Lemma~\ref{MainLemma}. Assume in addition that one of the following conditions holds:
\begin{enumerate}
\item[1)] the direct limit of the $A_k$ admits a decomposition of the form
$$
\varinjlim A_k \simeq D\oplus M
$$
where $D$ is a divisible group, 
and $M$ is a torsion group of bounded exponent;
\item[2)] the group $C[p^{\infty}]$ is torsion of bounded exponent.
\end{enumerate}
Then the exact sequence of direct limits \eqref{DLseq} splits.
\end{prop}

\begin{proof}
1) The exact sequence \eqref{DLseq} can be decomposed in two exact sequences
\begin{equation}
\label{LimSp}
\begin{CD}
0 @>>> D @>>> (\varinjlim B_k)' @>>> C[p^{\infty}] @>>> 0 \\
\end{CD}
\end{equation}
and
\begin{equation}
\label{LimSpp}
\begin{CD}
0 @>>> M @>>> (\varinjlim B_k)'' @>>> C[p^{\infty}] @>>> 0 \\
\end{CD}
\end{equation}
The sequence \eqref{LimSp} splits because $D$ is divisible, hence is an injective abelian group. The sequence \eqref{LimSpp} splits because $M$ being torsion of bounded exponent, this sequence is induced by one of the sequences~$\Seq_k$, which all split by Lemma~\ref{MainLemma}. Therefore, the sequence \eqref{DLseq} splits.

2) The group $C[p^{\infty}]$ being torsion of bounded exponent, there exists an integer $n$ such that $C[p^{\infty}]=C[p^n]$. The result then follows from Lemma~\ref{MainLemma}.
\end{proof}


\subsection{Dual version of the Main Lemma}

Our Lemma has a dual version as follows.

\begin{lem}[Dual of Main Lemma]
\label{CoMainLemma}
Let $p$ be a prime number, let $n>0$ be an integer, and let $A$ be an abelian group.

Assume that we have, for each $0<k\leq n$, an exact sequence $\Seq_k$ of $\Z/p^k$-modules
\begin{equation}
\tag{$\Seq_k$}
\begin{CD}
0 @>>> A/p^k @>>> B_k @>>> C_k @>>> 0 \\
\end{CD}
\end{equation}
together with a sequence of maps between these sequences
\begin{equation}
\begin{CD}
0 @>>> A/p^{k+1} @>>> B_{k+1} @>>> C_{k+1} @>>> 0 \\
@. @VVV @VVV @VVV \\
0 @>>> A/p^k @>>> B_k @>>> C_k @>>> 0 \\
\end{CD}
\end{equation}
where the vertical map on the left is the canonical surjection $A/p^{k+1}\twoheadrightarrow A/p^k$.

Then the exact sequence $\Seq_n$ splits.
\end{lem}

\begin{proof}
The group $\Q/\Z$ being injective, for each $0<k\leq n$ we obtain by applying the functor $\Hom(-,\Q/\Z)$ to $\Seq_k$ an exact sequence of $\Z/p^k$-modules
\begin{equation*}
\begin{CD}
0 @>>> \Hom(C_k,\Q/\Z) @>>> \Hom(B_k,\Q/\Z) @>>> \Hom(A/p^k,\Q/\Z) @>>> 0 \\
\end{CD}
\end{equation*}
that we denote by $\widehat{\Seq_k}$ (this is the Pontryagin dual of $\Seq_k$). We note that $\Hom(A/p^k,\Q/\Z)=\Hom(A,\Q/\Z)[p^k]$, and that the map $\Seq_{k+1}\to \Seq_k$ induces a map $\widehat{\Seq_k}\to\widehat{\Seq_{k+1}}$. Therefore, the hypothesis of Lemma~\ref{MainLemma} are satisfied by the family $\widehat{\Seq_k}$, and it follows that $\widehat{\Seq_n}$, hence $\Seq_n$, splits.
\end{proof}


\subsection{Proof of Theorem~\ref{coh_splitting}}

By the Chinese remainder theorem, it suffices to consider the case where $m$ is a power of some prime number $p$. We therefore let $m=p^n$. The map $[p^n]:N\to N$ being surjective, the same holds for the map $[p^k]:N\to N$ for each $0<k\leq n$.

For each $0<k\leq n$, we have an obvious commutative diagram with exact lines
$$
\begin{CD}
0 @>>> N[p^k] @>>> N @>[p^k]>> N @>>> 0 \\
@. @VVV @| @VV[p]V \\
0 @>>> N[p^{k+1}] @>>> N @>[p^{k+1}]>> N @>>> 0. \\
\end{CD}
$$
Applying cohomology on both lines yelds a commutative diagram with exact lines
$$
\begin{CD}
\dots\longrightarrow @. R^{i-1}F(N) @>\delta_i^k>> R^iF(N[p^k]) @>>> R^iF(N) @>[p^k]>> R^iF(N) @. \longrightarrow \cdots \\
@. @V[p]VV @VVV @| @VV[p]V \\
\dots\longrightarrow @. R^{i-1}F(N) @>\delta_i^{k+1}>> R^iF(N[p^{k+1}]) @>>> R^iF(N) @>[p^{k+1}]>> R^iF(N) @. \longrightarrow \cdots \\
\end{CD}
$$
where the $\delta_i$ are the connecting homomorphisms. One deduces a commutative diagram whose lines are short exact sequences
$$
\begin{CD}
0 @>>> R^{i-1}F(N)/p^k @>>> R^iF(N[p^k]) @>>> R^iF(N)[p^k] @>>> 0 \\
@. @VVV @VVV @VVV \\
0 @>>> R^{i-1}F(N)/p^{k+1} @>>> R^iF(N[p^{k+1}]) @>>> R^iF(N)[p^{k+1}] @>>> 0 \\
\end{CD}
$$
where the vertical map on the left is induced by the map $[p]:R^{i-1}F(N)\to R^{i-1}F(N)$, and the vertical map on the right is the natural inclusion. This proves that the sequence \eqref{cohomo_seq} fits into a family of sequences satisfying the hypothesis of Lemma~\ref{MainLemma}, hence splits.


\subsection*{Acknowledgments}

The first author was supported by the CIMI Excellence program while visiting the \emph{Centro di Ricerca Matematica Ennio De Giorgi} during the autumn of 2017. The second author was partially supported by CONICYT, Proyectos Regulares FONDECYT n\textsuperscript{o}~1150595, and by project number P27600 of the Austrian Science Fund (FWF).

The authors warmly thank Christian Wuthrich for providing the material of \S{}\ref{subsect_Chris}, and for sharing his valuable comments on our work. We also thank Bjorn Poonen for pointing out the paper \cite{BKLP15}, and for stimulating email exchange.





\bigskip

Jean Gillibert, Institut de Math{\'e}matiques de Toulouse, CNRS UMR 5219, 118, route de Narbonne, 31062 Toulouse Cedex 9, France.

\emph{Email address:} \texttt{jean.gillibert@math.univ-toulouse.fr}
\medskip

Pierre Gillibert, Institut f{\"u}r Diskrete Mathematik \& Geometrie, Technische Universit{\"a}t Wien, Wien,  {\"O}sterreich.

\emph{Email address:} \texttt{pgillibert@yahoo.fr}


\end{document}